\documentclass[a4paper,12pt]{amsart}

\usepackage{amsthm,amsmath,amssymb,latexsym} 
\usepackage{tikz}
\usetikzlibrary{positioning, arrows, automata}
\usepackage{etoolbox}
\usepackage{enumitem}
\usepackage{color}
\usepackage[colorlinks,linkcolor=blue,citecolor=magenta]{hyperref}
\usepackage[colorinlistoftodos,bordercolor=orange,backgroundcolor=orange!20,linecolor=orange,textsize=scriptsize]{todonotes}
\usepackage{soul}
\usepackage{microtype}
\usepackage{yhmath}
\usepackage{todonotes}
\usepackage{mathtools}
\usepackage{verbatim}
\usepackage{tikz-cd}
\usetikzlibrary{babel}
\usepackage{anysize}
\usepackage{soul,xcolor}
\hypersetup{
    colorlinks=true,
    linkcolor=red,
    filecolor=magenta,      
    urlcolor=cyan,
}

\newtheorem{theorem}{Theorem} 

\newtheorem{proposition}{Proposition}

\newtheorem{lemma}{Lemma}

\newtheorem*{proof-claim}{Proof}

\theoremstyle{definition}
\newtheorem{remark}{Remark}
\newtheorem{definition}{Definition}

\newtheorem{question}{Question}
\newtheorem{conjecture}{Conjecture}

\newtheoremstyle{colon}%
{}
{}
{\itshape}
{}
{\bfseries}
{:}
{ }
{}

\theoremstyle{colon}
\newtheorem*{statement1}{$\mathsf{HCT}_{n}(G)$}
\newtheorem*{statement2}{$\mathsf{HCX}_{n}(G)$}

\usepackage[margin=1in]{geometry}

\DeclareMathOperator{\ind}{ind}
\DeclareMathOperator{\coind}{co-ind}

\DeclareMathOperator{\xind}{x-ind}
\DeclareMathOperator{\sd}{sd}

\def\leq{\leqslant}
\def\geq{\geqslant}

\begin{document} 

\title{A topological version of Hedetniemi's conjecture for equivariant spaces}

\author{Vuong Bui}
\address{Vuong Bui, Institut f\"{u}r Informatik, Freie Universit\"{a}t Berlin, Takustrasse 9, 14195 Berlin, Germany}
\email{bui.vuong@yandex.ru}

\author{Hamid Reza Daneshpajouh}
\address{Hamid Reza Daneshpajouh,
School of Mathematical Sciences, University of
Nottingham Ningbo China, 199 Taikang East Road, Ningbo, 315100, China}
\email{Hamid-Reza.Daneshpajouh@nottingham.edu.cn}

\begin{abstract}
A topological version of the famous Hedetniemi conjecture says:
The mapping index of the Cartesian product of two $\mathbb Z/2$-spaces is equal to the minimum of their $\mathbb Z/2$-indexes. The main purpose of this article is to study  the topological version of the Hedetniemi conjecture for $G$-spaces. Indeed, we show that the topological Hedetniemi conjecture cannot be valid for general pairs of $G$-spaces. More precisely, we show that this conjecture can possibly survive if the group $G$ is either a cyclic $p$-group or a generalized quaternion group whose size is a power of 2.
\end{abstract}


\maketitle

\section{Introduction}\label{sec:intro}

The original motivation of this work comes from a long-standing conjecture of Stephen T. Hedetniemi~\cite{hedetniemi1966} which has been disproved  recently~\cite{shitov2019}. In $1966$, Hedetniemi conjectured that the chromatic number of the categorical product of two graphs is equal to the minimum of their chromatic numbers. This conjecture has attracted a great deal of interest over the past half-century. The conjecture has been shown to hold for some families of graphs~\cite{zhu1998survey, tardif2008hedetniemi, hajiabolhassan2016hedetniemi}, and also the fractional version of this conjecture has been verified~\cite{zhu2011fractional}. Although there were some some positive partial results in this regard, this longstanding conjecture has ended up being false with a counterexample given by Y. Shitov~\cite{shitov2019}. However, there are still some interesting open questions around. Not long before the conjecture got disproved, it had been shown~\cite{wrochna2019,matsushita2017z2} that if the conjecture held, then it would imply a similar equality in the category of equivariant space. To state it precisely, we need to recall the definition of mapping index for $G$-spaces.

Throughout this paper $G$ stands for a non-trivial finite group. For a $G$-space $X$ with a free action of a finite group $G$, the \emph{mapping index} $\ind X$ is the minimal $k$ such that there exists a $G$-equivariant map\footnote{A $G$-equivariant map $f: X\to Y$ is a continuous map that also preserves the $G$-action, i.e., $f(gx)=gf(x)$ for all $g\in G$ and $x\in X$. Moreover, if $X$ and $Y$ are $G$-simplicial complexes and $f$ is also a simplicial map, then it is called a $G$-simplicial map.} $X \to E_kG$, where $E_kG$ is the standard $(k+1)$-fold join $G^{*(k+1)}$. In the case $G=\mathbb Z/2$ the space $E_kG$ is topologically a sphere $S^k$ with the antipodal action of $\mathbb Z/2$, given by $x\mapsto -x$; hence, in this case, the definition is about equivariant maps to spheres. Now, we are in a position to recall the aforementioned topological statement. Indeed, they proved if Hedetniemi's conjecture is true, then the mapping index of the Cartesian product of two $\mathbb Z/2$-spaces (equipped with the diagonal action) is equal to the minimum of their $\mathbb Z/2$-indexes for every pair of finite $\mathbb Z/2$-simplicial complexes. Moreover, M. Wrochna~\cite{wrochna2019} conjectured the correctness of this statement. 
\begin{conjecture}[\cite{wrochna2019}]
\label{conj: product}
For every pair $\mathcal{K},\mathcal{L}$ of finite free $\mathbb Z/2$-simplicial complexes, we have 
\begin{equation} \label{eq:key-equation}
    \ind \mathcal{K}\times \mathcal{L}  = \min\{\ind \mathcal{L}, \ind \mathcal{L}\}.
\end{equation}
\end{conjecture} 

The second author et al~\cite{DANESHPAJOUH2023105721} fully confirmed the version of this conjecture for the homological index of $\mathbb Z/2$-spaces, also they established a slightly weaker form of this result for the free action of prime cyclic groups $\mathbb Z/p$ (with odd prime $p$). Moreover, they showed the generalized form of Conjecture~\ref{conj: product} is valid for the case when one of the factors is an $E_kG$-space. In fact, they verified the equality~\eqref{eq:key-equation} for every pair of $G$-spaces $X$ and $Y$ where $Y$ is a tidy space \footnote{A $G$-space $Y$ is called tidy if $\ind Y =\coind Y$, where $\coind Y$ is the maximum $k$ such that there exists a $G$-equivariant map from $E_kG$ to $Y$.}. So, it is natural to ask whether the equality~\eqref{eq:key-equation} is valid for every pair of $G$-spaces. Unfortunately, it turns out to be not the case for every $G$. To mention our main results in this direction, we need a definition.

\begin{definition}
A finite group $G$ is called a ``nice'' group if either it is a cyclic $p$-group or a generalized quaternion group \footnote{The generalized quaternion group is given by the presentation $Q_{4n}=\langle a,b : a^n=b^2, a^{2n}=1, b^{-1}ab=a^{-1}\rangle$ where $n\ge 2$.} whose size is a power of $2$.
\end{definition}
\begin{remark}
\label{rem:nice-group}
Actually, nice groups are classifications of all finite groups with a unique minimal non-trivial subgroup. Indeed, due to the classical Cauchy theorem such a group must be a $p$-group for some prime $p$, and then one can use~\cite[Theorem 4.10]{gorenstein2007finite} to verify this claim.
\end{remark}

Throughout this paper, for given $G$-spaces $X$ and $Y$, the Cartesian product $X\times Y$ is always considered as a $G$-space equipped with the diagonal action, i.e., $g\cdot (x,y)\mapsto (gx, gy)$. Now, we are in the position to mention the main result of this paper.
\begin{theorem}
\label{Thm: main}
If $G$ is not a nice group, then there are finite free $G$-simplicial complexes $\mathcal{K}_1,\mathcal{K}_2$ so that  $\ind \mathcal{K}_1=\ind \mathcal{K}_2=1$ but $\ind \mathcal{K}_1\times \mathcal{K}_2=0$.
\end{theorem}

Therefore, the generalized form of Conjecture \ref{conj: product} cannot be valid for every pair of $G$-spaces.
However, we may still hope that the conjecture is valid for every pair of $G$-spaces where $G$ is a nice group. To state this result precisely, first we introduce a statement that it will be used later as well.

\begin{statement1}
For every finite free $G$-simplicial complexes $\mathcal{K}_1$ and $\mathcal{K}_2$,
\[
    \textbf{if}\,\, \ind \mathcal{K}_1= \ind \mathcal{K}_2= n,\,\, \textbf{then}\,\, \ind\mathcal{K}_1\times \mathcal{K}_2 = n.
\]
\end{statement1}

\begin{remark}
\label{rem:new}
It is obvious that, for a fixed group $G$, the necessary condition for the topological Hedetniemi's conjecture being true is that the statement $\mathsf{HCT}_{n}(G)$ must be true for all $n\geq 0$. Actually, this condition is enough as well. To see this, first note that for every pair $\mathcal{K},\mathcal{L}$ of finite free $G$-simplicial complexes, we have 
\[
    \ind \mathcal{K}\times \mathcal{L}\leq \min\{\ind \mathcal{K}, \ind \mathcal{L}\},
\]
as the projection maps $\pi_1:\mathcal{K}\times \mathcal{L}\to \mathcal{K}$ and $\pi_2:\mathcal{K}\times \mathcal{L}\to \mathcal{L}$ are $G$-equivariant maps. So, if the topological Hedetniemi's conjecture is not true for a group $G$, then there are finite free $G$-simplicial complexes $\mathcal{K},\mathcal{L}$ such that 
\[
    \ind \mathcal{K}\geq\ind \mathcal{L}= n,\,\, \textbf{but}\,\, \ind\mathcal{K}\times \mathcal{L} < n,
\]
for some $n\geq 0$. If $\ind \mathcal{K} = n$, then the pair $\mathcal{K},\mathcal{L}$ shows $\mathsf{HCT}_{n}(G)$ is wrong. If not, that is $\ind \mathcal{K} > n$, then we can replace $\mathcal{K}$ with its an equivariant sub-complex $\mathcal{K}^{\prime}$ such that $\ind \mathcal{K}^{\prime}=n$. To see this, note that one can easily build a $G$-equivariant map from the zero-skeleton ${\mathcal{K}}_{0}$ of $\mathcal{K}$ to $G$, in other words $\ind {\mathcal{K}}_{0}=0$. On the other hand, the mapping-index can increases by at most one by passing from the $i$-skeleton ${\mathcal{K}}_{i}$ to $(i+1)$-skeleton ${\mathcal{K}}_{(i+1)}$ of $\mathcal{K}$~\cite[Lemma 11]{daneshpajouh2019}, i.e., $\ind{\mathcal{K}}_{(i+1)}-\ind{\mathcal{K}}_{i}\leq 1$. Thus, there is an $0\leq i\leq \dim \mathcal{K}$ such that $\ind {\mathcal{K}}_{i}=n$. Set, $\mathcal{K}^{\prime}= {\mathcal{K}}_{i}$. Now, the pair $\mathcal{K}^{\prime}, \mathcal{L}$  shows that $\mathsf{HCT}_{n}(G)$ cannot be valid which this verifies the claim. 
\end{remark}
Now, the following result can serve as an evidence that the topological Hedetniemi's conjecture for nice groups might be plausible.
\begin{theorem}
\label{Thm: main2}
Let $G$ be a nice group. If $\mathcal{K}_1,\mathcal{K}_2$ are $G$-simplicial complexes so that  $\ind \mathcal{K}_1=\ind \mathcal{K}_2=1$, then $\ind\mathcal{K}_1\times\mathcal{K}_2=1$. In other words, $\mathsf{HCT}_{1}(G)$ is true.
\end{theorem}
This result was known for the case $G=\mathbb{Z}_2$~\cite{wrochna2019,matsushita2017z2}. It is also worth pointing out that the proofs of Theorems~\ref{Thm: main},~\ref{Thm: main2} are based on a combinatorial analogue of mapping-index, which is called cross-index. In order to define cross-index and also mention our last result, we need some definitions.  

A $G$-poset $(P, \preceq)$ is a partially ordered set with an order preserving  $G$-action on its ground set, i.e.,
$p_1\preceq p_2$ implies $gp_1\preceq gp_2$ for all $g\in G$ and $p_1, p_2\in P$. A $G$-poset is said to be free if $gp=p$ implies $g=e$ for all $p\in P$ and $g\in G$. A $G$-map $\psi : P\to Q$ between $G$-posets $P$ and $Q$ is an order-preserving map, i.e., $\psi (p_1)\preceq \psi(p_2)$ if $p_1\preceq p_2$, which also preserves the action, that is $\psi (gp) = g\psi (p)$ for all $p\in P$ and $g\in G$. For an integer $n\geq 0$, let $Q_nG$ be the $G$-poset on the ground set $G\times\{0,\ldots , n\}$,
with its natural $G$-action, $g\cdot (h, i)\mapsto (gh, i)$, and the order defined by $(g,i)\prec (h, j)$ if $i < j$ in $\mathbb{N}$.
\begin{definition}\footnote{It should be noted that the cross-index for $G$-poset where $G=\mathbb{Z}_2$ and $\mathbb Z/p$ were defined respectively in~\cite{simonyi2013colourful} and~\cite{alishahi2017colorful}.}
For a $G$-poset $P$, the cross-index of $P$, denoted by $\xind P$, is the smallest $n$ such that $P$ admits a $G$-map to $Q_nG$.
\end{definition}
If $P$ and $Q$ are posets, then the product $P\times Q$ is the poset whose elements are all $(p, q)$ such that $p\in P$ and $q\in Q$ and $(p_1,q_1)\preceq (p_2,q_2)$ if $p_1\preceq p_2$ and $q_1\preceq q_2$. Moreover, if $P$ and $Q$ are $G$-poset, then $P\times Q$ is a $G$-poset with the diagonal action, i.e., $g\cdot(p,q)\to (gp, gq)$ for all $(p,q)\in P\times Q$ and all $g\in G$. The face poset $\mathcal{F}(\mathcal{K})$ of a simplicial complex $\mathcal{K}$ is the poset whose vertices are all non-empty simplicies of $\mathcal{K}$ ordered with the inclusion. If $\mathcal{K}$ is a $G$-simplicial complex, then we consider $\mathcal{F}(\mathcal{K})$ as a $G$-poset with the action naturally induced from $\mathcal{K}$. Finally, similar to $\mathsf{HCT}_{n}(G)$, we define an analogues statement for special family of $G$-poset. 
  
\begin{statement2}
    For every finite free $G$-simplicial complexes $\mathcal{K}_1$ and $\mathcal{K}_2$,
\[
    \textbf{if}\,\, \xind\mathcal{F}(\mathcal{K}_1)=\xind\mathcal{F}(\mathcal{K}_2)=n,\,\, \textbf{then}\,\, \xind\mathcal{F}(\mathcal{K}_1)\times \mathcal{F}(\mathcal{K}_2)= n.
\]
\end{statement2}

Now we are in a position to mention our final result.

\begin{theorem}
\label{Thm: main3}  
For every finite group $G$ and non-negative integer $n\geq 0$, $\mathsf{HCX}_{n}(G)$ implies $\mathsf{HCT}_{n}(G)$. 
\end{theorem}

The organization of the paper is as follows. In Section \ref{sec:cross-index} we discuss the connection between mapping-index and cross-index in more details. Section \ref{sec:proofs} is devoted to the proofs of our main results, i.e., Theorems~\ref{Thm: main},~\ref{Thm: main2}, and~\ref{Thm: main3}. Finally, in the last section we present some open problems.

\section{Cross-Index: a combinatorial analogue of mapping-index}
\label{sec:cross-index}

Here and subsequently, for a given positive integer $r$, the $r$-barycentric subdivision of a simplicial complex $\mathcal{K}$ is denoted by $\sd^r(\mathcal{K})$. We also set $\sd^0(\mathcal{K})=\mathcal{K}$. For a poset $P$ its order complex $\Delta(P)$ is the simplicial complex whose simplicies are all non-empty chains in $P$. If $P$ is a $G$-poset, then we consider $\Delta(P)$ as a $G$-simplicial complex with the induced $G$-action from $P$. Note that any $G$-map $\psi : P\to Q$ between two $G$-posets induces a simplicial $G$-map $\Delta(\psi): \Delta P\to \Delta Q$. Hence, by considering the definitions of 
cross-index, mapping-index and the fact that $\Delta Q_n\cong_{G}E_nG$, i.e., $\Delta Q_n$ is $G$-homeomorphic to $E_nG$, we have 
\begin{equation}
\label{eq: cross vs mapping-index}
   \ind \Delta P \leq \xind P.
\end{equation}

Inequality \eqref{eq: cross vs mapping-index} can be tight, and the following proposition is one particular case.
\begin{proposition}
 \label{lem: xind=0}
 For every free $G$-poset $P$, we have 
\[
    \xind P=0\quad \Longleftrightarrow\quad \ind \Delta (P) = 0.
\]
\end{proposition}

\begin{proof}
  If $\xind(P)=0$ then $\ind \Delta(P) = 0$  by Inequality~\eqref{eq: cross vs mapping-index}. For the other direction, if $\ind \Delta(P) = 0$, then there is a $G$-equivariant map $\psi : \Delta(P)\to G$. The map $\psi$ sends each (path)-connected component of $\Delta(P)$ to a single point of $G$ as $\psi$ is continuous and $G$ has the discrete topology. This shows that the natural induced map $\bar{\psi}: P\to G\times[0]$, which sends $p$ to $(\psi(p), 0)$, is an order preserving as any two comparable elements in $P$ lies in a same path-component of $\Delta(P)$. Clearly $\bar{\psi}$ preserves the $G$-action as $\psi$ does. Hence, $\bar{\psi}$ is a $G$-map, and therefore $\xind P=0$. Now, the proof is complete.
\end{proof}

From a computational viewpoint, deciding whether the cross-index of a given $G$-poset is zero is an ``easy task''. Indeed, the purpose of next proposition is to establish this fact. Before that, let us remind the definition of comparability graph.
\begin{definition}
    The comparability graph of a poset $P$ is an \emph{undirected} graph whose vertices are elements of $P$ and there is an edge between vertices $u,v$ if and only if $u$ and $v$ are comparable in $P$, i.e., $u\preceq v$ or $v\preceq u$.
\end{definition}

\begin{proposition} \label{prop:no-path}
    For any finite free $G$-poset $P$, we have $\xind P=0$ if and only if there is no path between two elements of the same orbit in the comparability graph.
\end{proposition}
Note that this result is already known for the case $G=\mathbb Z_2$ (see the proof of \cite[Theorem 9]{simonyi2013colourful}).
\begin{proof}
    Suppose there is no such path in the comparability graph, we prove $\xind P=0$ by establishing a valid $G$-map $\psi: P\to Q_0G$. 
    At first, we take an arbitrary element $x_0\in P$ and assign $\psi(x_0)=(e,0)$. Then there is a unique way to extend the map to the orbit $[x_0]$ of $x_0$ so that the $G$-action is preserved, i.e., $\psi(gx_0)=(g,0)$ for any $g\in G$. 
    Note that the elements $x$ of $[x_0]$ lie in different components $\Gamma_x$ of the comparability graph of $P$, due to the condition on the paths.
    Now, for each component $\Gamma_x$, we assign $\psi(y)=\psi(x)$ for every $y\in\Gamma_x$.
    If there is any element of $P$ that has not been assigned yet, we continue the same procedure for such an element, and recursively do it until there is no remaining unassigned element. The final function $\psi$ is a valid $G$-map since the $G$-action and the order are both preserved. 

    For the other direction, first note that any order-preserving $\psi: P\to Q_0G$ must be constant on each component of the comparability graph of $P$. So, if two distinct elements of the same orbit lie in the same component of $P$, then such a map cannot preserve the action anymore and hence it is not a $G$-map. Therefore, there is no path between two elements of the same orbit when the cross-index is $0$.
\end{proof}

We should note that Inequality \eqref{eq: cross vs mapping-index} is not tight in general. For the case $G=\mathbb Z/2$ see the last remark in~\cite{simonyi2013colourful}. However, if we subdivide $\Delta P$ enough, then the cross-index of the face poset of that refinement matches with the mapping-index of $\Delta P$. To verify this claim, let us start with the following easy observation that it is needed for the proof.

\begin{proposition}
\label{lem:1}
For any finite $G$-poset $P$ and any $r\geq 0$, there is a $G$-map from 
$\mathcal{F}\left(\sd^r(\Delta(P))\right)$ to $P$.
\end{proposition}
\begin{proof}
It suffices to show this for $r=0$. Define $\psi : \mathcal{F}(\Delta(P))\to P$ 
by sending each $A\in \sd (P)$ to the maximum element of $A$. It is easy to check that $\psi$ is a $G$-map.
\end{proof}
\begin{proposition}\label{Thm: simplicial approximation theorem}
For each finite free $G$-poset $P$, there is an $r_0\geq 0$ such that for any $r\geq r_0$:
\[
    \xind\mathcal{F}\left(\sd^r(\Delta(P))\right) = \ind \bigtriangleup (P).
\]
\end{proposition}
\begin{proof}
First note that, Proposition~\ref{lem:1} shows that the sequence $\{\xind\mathcal{F}(\sd^n (\Delta(P))\}_{n}$
is a decreasing sequence. Inequality~\eqref{eq: cross vs mapping-index} implies that each term of this sequence is bounded from below by $\ind\Delta (P)$ as we have $\ind\sd^{n} (\Delta(P))=\ind \Delta(P)$ for every $n\geq 0$. The latter claim follows from that fact that every $G$-simplicial complex is $G$-homeomorphic to its barycentric subdivision. 

Now, let $\ind \Delta (P) = m$. So, there is a $G$-map $\psi : \Delta (P)\to\ E_mG\cong_{G}\Delta Q_m$. By the equivariant version of simplicial approximation theorem, there exists an $r_0\geq 0$ and a $G$-equivariant simplicial
map $g : \sd^{r_0}(\Delta (P))\to \Delta (Q_m)$. This map induces a $G$-map from $\mathcal{F}(\sd^{r_0}(\Delta (P))$ into $\mathcal{F}( \Delta (Q_m))$. But by Proposition \ref{lem:1}, there is a $G$-map from $\mathcal{F}( \Delta (Q_m))$ to $Q_m$. Combining these $G$-maps defines a $G$-map from $\mathcal{F}(\sd^{r_0}(\Delta (P)$ to $Q_m$, and hence $\xind \mathcal{F}(\sd^{r_0} (\Delta(P))\leq \ind\Delta (P)$ which implies $\xind \mathcal{F}(\sd^r (\Delta(P))=\ind\Delta (P)$ for every $r\geq r_0$ by the earlier claimed established in the beginning of the proof.
\end{proof}
\section{Proofs of Main Results}
\label{sec:proofs}
Before proving Theorem \ref{Thm: main}, we need the following lemma, which also shows that the Hedetniemi conjecture for cross-index is not true in general.
\begin{lemma}
\label{lem: main lemma}
    If $G$ is not a nice group, then there are finite free $G$-posets $P_1,P_2$ with $\xind P_1=\xind P_2=1$ but $\xind P_1\times P_2 =0$.
\end{lemma}
\begin{proof}
    Since $G$ is not a nice group, it contains two minimal nontrivial subgroups $H_1,H_2$ (see Remark \ref{rem:nice-group}). Since $H_1,H_2$ are minimal and nontrivial, there are non-identity elements $h_1\in H_1,h_2\in H_2$ so that $H_1,H_2$ are generated by $h_1,h_2$, respectively. Note that the intersection of $H_1$ and $H_2$ is trivial. 
    
    Let $P_1$ be the $G$-poset whose set of elements is $G^{(1)}\cup G^{(2)}$ where each of $G^{(1)}, G^{(2)}$ is a copy of $G$ with the order defined as follows. First, denote by $g^{(i)}$ the corresponding element to $g\in G$ in $G^{(i)}$ for $i=1,2$.
    Then, we let $g^{(1)}\preceq g^{(2)}$ for each $g\in G$. After that, we let $e^{(1)}\preceq {h_1}^{(2)}$ and extend it minimally, that is we let $ge^{(1)}\preceq gh_1^{(2)}$ for any $g\in G$. (Note that the $G$-action is the natural one with $gh^{(i)} = (gh)^{(i)}$ for any $g,h\in G$ and $i=1,2$.)
    
    We construct $P_2$ in the same way except that we extend $e^{(1)}\preceq {h_2}^{(2)}$ instead. In Fig. \ref{fig:main-lemma}, we illustrate an example of $P_1,P_2$ with the group $G=\mathbb Z_2\times\mathbb Z_2=\{\underbrace{(0,0)}_{e},\underbrace{(1,0)}_{h_1},\underbrace{(0,1)}_{h_2},\underbrace{(1,1)}_{h_3}\}$.
    
    \begin{figure}[ht]
    \centering
    \includegraphics[width=0.8\textwidth]{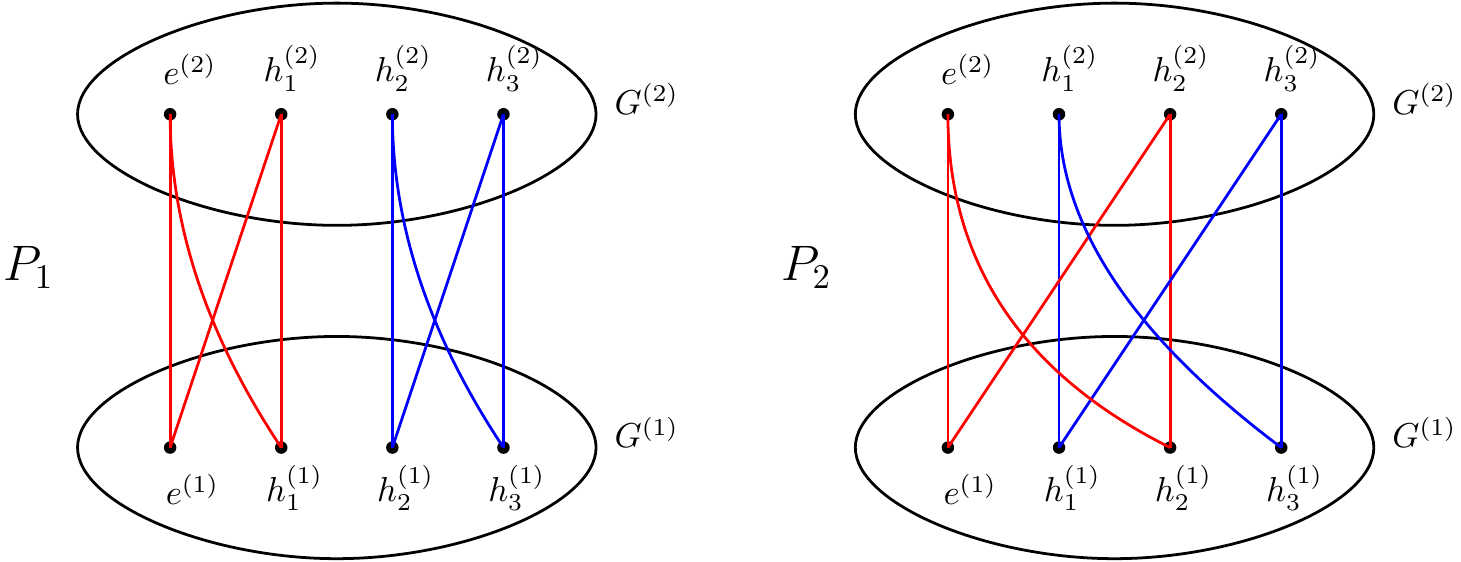}
    \caption{An example of two posets $P_1,P_2$ in the proof of Lemma \ref{lem: main lemma}}
    \label{fig:main-lemma}
    \end{figure}

    By Proposition~\ref{prop:no-path}, the cross-indices of $P_1, P_2$ are nonzero, because the path $e^{(1)}\to {h_i}^{(2)}\to {h_i}^{(1)}$ connects two elements $e^{(1)}, {h_i}^{(1)}$ of the same orbit for $i=1,2$.
    
    Furthermore, we show that $\xind P_1=\xind P_2=1$. Indeed, the function $\psi_j: P_j\to Q_1G$ that sends each $g^{(i)}$ to $(g, i-1)$ is a valid $G$-map for $j=1,2$. (In fact, the same conclusion holds for any poset of two orbits.)

    It remains to prove $\xind P_1\times P_2=0$.
    Assume otherwise that $\xind P_1\times P_2\ge 1$, we derive a contradiction.   By Proposition \ref{prop:no-path}, there is a path
    \[
        ({x_1}^{(i_1)},{y_1}^{(j_1)})\to ({x_2}^{(i_2)},{y_2}^{(j_2)})\to \dots\to ({x_k}^{(i_k)},{y_k}^{(j_k)})
    \]
    between two elements of the same orbit in the comparability graph.
    Since ${x_t}^{(i_t)}$ and ${x_{t+1}}^{(i_{t+1})}$ are comparable for every $t$, we have either $x_{t+1}=x_t$, or $x_{t+1}=h_1 x_t$, or $x_{t+1}=h_1^{-1} x_t$.
    That is $x_1$ and $x_k$ lie in the same coset of $H_1$.
    Likewise, $y_1$ and $y_k$ lie in the same coset of $H_2$. 
    Let $x_k=h' x_1$ and $y_k=h'' y_1$, where $h'\in H_1$ and $h''\in H_2$. As $({x_1}^{(i_1)},{y_1}^{(j_1)}),({x_k}^{(i_k)},{y_k}^{(j_k)})$ are in the same orbit, we have $h'=h''$. Since the intersection of $H_1$ and $H_2$ is trivial, we obtain $h'=h''=e$. However, it means $({x_1}^{(i_1)},{y_1}^{(j_1)}),({x_k}^{(i_k)},{y_k}^{(j_k)})$ are identical, contradiction.
\end{proof}

We can now prove Theorem \ref{Thm: main}.
\begin{proof}[Proof of Theorem \ref{Thm: main}]
 Let $P_1$ and $P_2$ be the $G$-posets which were defined in Lemma~\ref{lem: main lemma}. Set $\mathcal{K}_1=\Delta(P_1), \mathcal{K}_2=\Delta(P_2)$. Proposition~\ref{lem: xind=0} and Inequality~\eqref{eq: cross vs mapping-index} imply $\ind \mathcal{K}_1=\ind \mathcal{K}_2=1$. On the other hand, $\mathcal{K}_1\times \mathcal{K}_2= \Delta(P_1)\times \Delta(P_2)$ is $G$-homeomorphic to $\Delta (P_1\times P_2)$~\cite{walker1988canonical}
 which implies
 \[
     \ind \mathcal{K}_1\times \mathcal{K}_2=\ind \Delta (P_1\times P_2)\leq \xind P_1\times P_2=0.\qedhere
 \]
\end{proof}

In order to prove Theorem \ref{Thm: main2}, we again start with a combinatorial version.
\begin{lemma}
\label{lem: product-cross}
	Let $G$ be a nice group. If $P$ and $Q$ are finite free $G$-posets with $\xind P =\xind Q=1$, then $\xind P\times Q=1$. In particular, $\mathsf{HCX}_{1}(G)$ is true.
\end{lemma}
\begin{proof}

	Since $\xind P=1, \xind Q=1$, it follows from Proposition \ref{prop:no-path} that there exist two paths $p\to\dots\to gp$ and $q\to \dots\to hq$ in the comparability graphs of $P$ and $Q$, respectively, for some $p\in P$, $g\in G, g\ne e$ and some $q\in Q$, $h\in G, h\ne e$.
	
	As $G$ is a nice group, the subgroups generated by $g$ and $h$ share some element $g^*\ne e$.  
We now construct a path from $p_1$ to $g^*p_m$. 
Multiply the path $p\to\dots\to gp$ by $g$, we obtain $gp\to\dots\to g^2 p$. Repeating the same procedure, the path can be extended to
\[
    p\to\dots\to gp\to\dots\to g^2 p\to\dots\to g^tp
\]
for any $t\ge 1$.
Let $t$ be so that $g^t=g^*$, then we obtain a path from $p$ to $g^*p$. By a similar construction, we also obtain a path from $q$ to $g^*q$. Let $L$ denote the path from $p$ to $g^*p$ and $L'$ denote the path from $q$ to $g^*q$. The concatenation of $p\times L'$ and $L\times g^*q$, i.e.
 \[
   (p,q)\to\dots\to (p,g^*q)\to\dots\to (g^*p,g^*q),
 \]
is actually a path from $(p,q)$ to $(g^*p,g^*q)$ in the comparability graph of $P\times Q$. 
 Since $(p,q)$ and $(g^*p,g^*q)$ are two distinct elements in the same orbit, it follows from Proposition \ref{prop:no-path} that $\xind P\times Q > 0$. In fact, $\xind P\times Q=1$ because the projection map $P\times Q\to P$ on the first component is a $G$-map which in particular implies $\xind (P\times Q)\le \xind(P)=1$.
\end{proof}

Now, we are in position to present the proofs of Theorems~\ref{Thm: main2} and~\ref{Thm: main3}.

\begin{proof}[Proof of Theorem \ref{Thm: main2}]
It is a direct consequence of Theorem~\ref{Thm: main3} and Lemma~\ref{lem: product-cross}.
\end{proof}
 So, to finish this section we need to provide a proof for Theorem~\ref{Thm: main3}. The idea of proof is similar to the proof of~\cite[Theorem 1.2]{matsushita2017z2}.
 
Note that, in general the Cartesian product $\mathcal{K}_1\times\mathcal{K}_2$ of two simplicial complexes $\mathcal{K}_1, \mathcal{K}_2$ is not a simplicial complex. This fact introduces some difficulties in the study of the mapping-index of $ \mathcal{K}_1\times\mathcal{K}_2$ by looking at the cross-index of the face poset of some subdivision of $\mathcal{K}_i$. But, fortunately, there is a notion of product of simplicial complexes that can be very beneficial for our purpose: The simplicial product $\mathcal{K}_1\boxtimes\mathcal{K}_2$ is a simplicial complex whose vertices are the pairs $(x_1, x_2)$ where $x_i$ is the vertex of $\mathcal{K}_i$ for $i=1,2$ and whose simplicies are all $A\subseteq V(\mathcal{K}_1)\times V(\mathcal{K}_2)$ such that $\pi_i(A)$ is a simplex of $\mathcal{K}_i$ for $i=1,2$ where $\pi_i: V(\mathcal{K}_1)\times V(\mathcal{K}_2)\to V(\mathcal{K}_i) $ is the projection map on the $i$-th component for $i=1,2$. Note that the projection map $\pi_i: V(\mathcal{K}_1)\times V(\mathcal{K}_2)\to V(\mathcal{K}_i)$ induces the natural simplicial map $p_i: \mathcal{K}_1\boxtimes\mathcal{K}_2\to\mathcal{K}_i$ for $i=1,2$.

In general, the simplicial product $\mathcal{K}_1\boxtimes\mathcal{K}_2$ does not provide a triangulation for the Cartesian product $\mathcal{K}_1\boxtimes\mathcal{K}_2$, i.e., these two spaces are not homeomorphic. However, they are homotopy equivalent. Indeed, it is known that the natural map $p: \mathcal{K}_1\boxtimes\mathcal{K}_2\to \mathcal{K}_1\times\mathcal{K}_2$, the map which sends $x$ in $\mathcal{K}_1\boxtimes\mathcal{K}_2$ to $(p_1(x),p_2(x))$ in $\mathcal{K}_1\times\mathcal{K}_2$, is homotopy equivalence ~\cite[Lemma 8.11]{eilenberg2015foundations}. This fact is almost enough for our purpose. The only problem is that we need the equivariant version of this fact. The $\mathbb Z/2$-equivariant version has been already established~\cite[Proposition 4.2]{matsushita2017z2} in the literature and a similar argument shows the $G$-equivariant version is also valid for any finite group $G$.
\begin{proposition}
\label{pro: the natural map is homotopy}
Let $\mathcal{K}_1, \mathcal{K}_2$ be free $G$-simplicial complexes. The natural map $p: \mathcal{K}_1\boxtimes\mathcal{K}_2\to \mathcal{K}_1\times\mathcal{K}_2$ is a $G$-homotopy equivalence.    
\end{proposition}
One can also deduce the equivariant version, Proposition~\ref{pro: the natural map is homotopy}, from the topological version using a theorem of Bredon~\cite[Section II.2]{tom2011transformation} which says when a $G$-equivariant map can be a $G$-homotopy equivalence. In particular, in the case the action is free, it says that a $G$-equivariant map is a $G$-homotopy equivalence if and only if it is an ordinary homotopy equivalence. 
\begin{proposition}
\label{pro: equivariant simplicial approximation theorem for product}
Let $\mathcal{K}, \mathcal{L}, \mathcal{M}$ be free $G$-simplicial complexes and $f: \mathcal{K}\times\mathcal{L}\to  \mathcal{M}$ be an $G$-equivariant map. Then, there is an integer $r\geq 0$ and a $G$-simplicial map $\psi : \sd^{r}(\mathcal{K})\boxtimes \sd^{r}(\mathcal{L})\to \mathcal{M}$ making the following diagram commute up to $G$-homotopy.

\begin{center}
\begin{tikzpicture}

\node (a) at (0,0)  {$\sd^{r}(\mathcal{K})\boxtimes \sd^{r}(\mathcal{L})$};
\node (b) at (5,0)  {$\sd^{r}(\mathcal{K})\times \sd^{r}(\mathcal{L})$};
\node (c) at (9,0)  {$\mathcal{K}\times\mathcal{L}$};
\node (d) at (9,-2)  {$\mathcal{M}$};
\node (x) at (2.5,0.3) {$p$};
\node (y) at (7.5,0.3) {$\cong_{G}$};
\node (z) at (9.3,-1) {$f$};
\node (w) at (5.5,-1.5) {$\psi$};

\draw[->,] (a) -- (b);
\draw[->] (b) -- (c);
\draw[->] (c) -- (d);
\draw[->] (a) -- (d);

\end{tikzpicture}
\end{center}

\end{proposition}
Again, this proposition is known~\cite[Proposition 4.2]{matsushita2017z2} for the case $G=\mathbb Z/2$ and one can use a similar argument to establish it for any finite group $G$.

\begin{proposition}
\label{pro: faceposet-products}
For every finite $G$-simplicial complexes $\mathcal{K}_1$ and $\mathcal{K}_2$, there is a $G$-map from $\mathcal{F}(\mathcal{K}_1)\times\mathcal{F}(\mathcal{K}_2)$ to $\mathcal{F}(\mathcal{K}_1\boxtimes \mathcal{K}_2)$ and vice versa. 
\end{proposition}
\begin{proof}
    Following maps do the job.
 \begin{align*}
\varphi:\mathcal{F}(\mathcal{K}_1)\times \mathcal{F}(\mathcal{K}_2) &\longrightarrow \mathcal{F}(\mathcal{K}_1\boxtimes \mathcal{K}_2)\\
  {\left(A,B\right)} &\longmapsto  {\left(A\times B\right)},
\end{align*}
and
 \begin{align*}
  \psi: \mathcal{F}(\mathcal{K}_1\boxtimes \mathcal{K}_2) &\longrightarrow \mathcal{F}(\mathcal{K}_1)\times \mathcal{F}(\mathcal{K}_2)\\
  {A}&\longmapsto {\left(\pi_1(A), \pi_2(A)\right)}. \qedhere
\end{align*}
\end{proof}

Now, we are in a position to present the proof of Theorem~\ref{Thm: main3}.

\begin{proof}[Proof of Theorem \ref{Thm: main3}]
Suppose $\mathsf{HCX}_n(G)$ is true for some $n\geq 0$. Let $\mathcal{K}_1$ and $\mathcal{K}_2$ be free $G$-simplicial complexes with $\ind \mathcal{K}_1= \ind \mathcal{K}_2=n$. Set $\ind\mathcal{K}_1\times \mathcal{K}_2 = m$. We need to show that $m = n$. As it is discussed in Remark~\ref{rem:new}, clearly we have $m\leq n$. To show that the other direction, note that there is a $G$-equivariant map $f: \mathcal{K}_1\times \mathcal{K}_2\to E_mG\cong_{G}\Delta(Q_mG)$ as $\ind\mathcal{K}_1\times \mathcal{K}_2 = m$. Now, by proposition~\ref{pro: equivariant simplicial approximation theorem for product}, there is an $r_1\geq 0$ and a simplicial $G$-map $\sd^{r_1}(\mathcal{K}_1)\boxtimes \sd^{r_1}(\mathcal{K}_2)\to \Delta(Q_mG)$. This, using Proposition~\ref{pro: faceposet-products} and Proposition~\ref{lem:1}, implies a $G$-map $\phi :\mathcal{F}(\sd^{r_1}(\mathcal{K}_1))\times \mathcal{F}\sd^{r_1}(\mathcal{K}_2)\to Q_mG$. Also, by Proposition~\ref{Thm: simplicial approximation theorem} there is a non-negative integer $r_2\geq r_1$ such that
\begin{align}
\label{eq:1}
  \ind\mathcal{K}_1 & = \xind\mathcal{F}(\sd^{r_2}(\mathcal{K}_1))\\
\label{eq:2} \ind\mathcal{K}_2 & = \xind\mathcal{F}(\sd^{r_2}(\mathcal{K}_2)).
\end{align}
Moreover, by Proposition~\ref{lem:1}, there are $G$-maps $\mathcal{F}(\sd^{r_2}(\mathcal{K}_1))\to \mathcal{F}(\sd^{r_1}(\mathcal{K}_1)) $ and $\mathcal{F}(\sd^{r_2}(\mathcal{K}_1))\to \mathcal{F}(\sd^{r_1}(\mathcal{K}_2)) $ as $r_2\geq r_1$. This implies a $G$-map 
\[\mathcal{F}(\sd^{r_2}(\mathcal{K}_1))\times \mathcal{F}(\sd^{r_2}(\mathcal{K}_2)) \to \mathcal{F}(\sd^{r_1}(\mathcal{K}_1))\times \mathcal{F}(\sd^{r_1}(\mathcal{K}_2)).\]
Finally, combining this map with $\phi :\mathcal{F}(\sd^{r_1}(\mathcal{K}_1))\times \mathcal{F}\sd^{r_1}(\mathcal{K}_2)\to Q_mG$ gives a $G$-map
\begin{align}
\label{eq:3}
    \mathcal{F}(\sd^{r_2}(\mathcal{K}_1))\times \mathcal{F}\sd^{r_2}(\mathcal{K}_2)\to Q_mG
\end{align}
Thus, we have
\begin{align*}
n & = \min\{\ind\mathcal{K}_1, \ind\mathcal{K}_2\}\\
& = \min\{\xind\mathcal{F}(\sd^{r_2}(\mathcal{K}_1)), \xind\mathcal{F}(\sd^{r_2}(\mathcal{K}_2))\} & \text{(by Equalities~\eqref{eq:1}, \eqref{eq:2}})\\
& = \xind\mathcal{F}(\sd^{r_2}(\mathcal{K}_1))\times\mathcal{F}(\sd^{r_2}(\mathcal{K}_2)) & \text{(by $\mathsf{HCX}_{n}(G)$)}\\
& \leq m & \text{by~\eqref{eq:3}}. &\qedhere
\end{align*}
\end{proof}

\section{Open Problems}
\label{sec:open-problems}
In this section we present some open problems and conjectures that we were not able to answer. According to our result in this paper, we believe the following conjectures might be true. We mention our conjectures from the strongest to the weakest form. 
\begin{conjecture}
If $G$ is a nice group, then for every finite free $G$-simplicial posets $\mathcal{P}_1$ and $\mathcal{P}_2$,
\[
    \xind P_1\times P_2=\min\{\xind P_1, \xind P_2\}.
\]
\end{conjecture}
\begin{conjecture}
If $G$ is a nice group, then for every finite free $G$-simplicial complexes $\mathcal{K}_1$ and $\mathcal{K}_2$,
\[
    \xind\mathcal{F}(\mathcal{K}_1)\times \mathcal{F}(\mathcal{K}_2)=\min\{\xind\mathcal{F}(\mathcal{K}_1), \xind\mathcal{F}(\mathcal{K}_2)\}.
\]
\end{conjecture}

\begin{conjecture} \label{conj:general-hedet}
If $G$ is a nice group, then for every finite free $G$-simplicial complexes $\mathcal{K}_1$ and $\mathcal{K}_2$,
\[
    \ind\mathcal{K}_1\times \mathcal{K}_2=\min\{\ind\mathcal{K}_1, \ind\mathcal{K}_2\}.
\]
\end{conjecture}

We have seen that the mapping-index of a $G$-simplicial complex 
is bounded above by the cross-index of its face poset. So, it is natural to ask how good this bound is. In particular, we are interested in the following questions.

\begin{question}
\label{que: que2}
Given positive integers $m$ and $n$ with $m< n$, is there any finite free $G$-posets $P$ such that
$\ind\Delta(P)  = m$ but $\xind P= n$? 
\end{question}

\begin{question}
\label{que: que1}
Given positive integers $m$ and $n$ with $m < n$, is there any finite free $G$-simplicial complex $\mathcal{K}$ such that
$\ind \mathcal{K}  = m$ but $\xind(\mathcal{F}(\mathcal{K}))= n$?
\end{question}

Regarding to Proposition~\ref{Thm: simplicial approximation theorem}, it is interesting to know how many times it is needed to subdivide a given $G$-simplicial complex $\mathcal{K}$ in which the mapping-index of $\mathcal{K}$ match with the cross-index of the face poset of that refinement of $\mathcal{K}$.

\begin{question}
\label{que: que5}
For a given free $G$-simplicial complex $\mathcal{K}$, what is the minimum integer $r \geq 0$ such that
\[
    \ind \mathcal{K} = \xind\mathcal{F}(\sd^{r}(\mathcal{K}))?
\]
\end{question} 

\begin{question}
\label{que: que5'}
For a given $r \geq 0$, is there a free $G$-simplicial complex $\mathcal{K}$ such that
\[
    \ind \mathcal{K} < \xind\mathcal{F}(\sd^{r}(\mathcal{K}))?
\]
\end{question}

\bibliographystyle{plain}
\bibliography{main}

\end{document}